\documentclass[a4paper,reqno,11pt]{amsart}
\usepackage{graphicx}
\usepackage{subfigure}
\usepackage{tikz, pgfplots}
\usepackage{hyperref}
\usepackage{epstopdf}

\hypersetup{
colorlinks=true,
linkcolor=blue,
anchorcolor=blue,
citecolor=blue
}

\usepackage{amssymb}
\usepackage{amsfonts}
\usepackage{amsmath}

\newtheorem{theorem}{Theorem}[section]
\newtheorem{lemma}[theorem]{Lemma}

\theoremstyle{definition}
\newtheorem{definition}[theorem]{Definition}

\newtheorem{remark}[theorem]{Remark}

\numberwithin{equation}{section}

\begin{document}
\title[Liouville theorems]{Liouville theorems to system of elliptic differential inequalities on the Heisenberg group}
\author[Yadong Zheng]{Yadong Zheng}
\address{School of Mathematical Sciences and LPMC, Nankai University, 300071
Tianjin, P. R. China}
\email{zhengyd@mail.nankai.edu.cn}

%\thanks{Xu was supported by the Fundamental Research Funds for the Central Universities, SYSU (No.20lgpy154)}


%\date{}

%\dedicatory{This paper is dedicated to our advisors.}

\begin{abstract}
In this paper, we establish Liouville theorems for the following system of elliptic differential inequalities
\begin{flalign}
 && &
 \begin{cases}
\Delta_{\mathbb H}u^{m_1}+|\eta|_{\mathbb H}^{\gamma_1}|v|^p\leq0,\\
\Delta_{\mathbb H}v^{m_2}+|\eta|_{\mathbb H}^{\gamma_2}|u|^q\leq0,
 \end{cases}
 &\nonumber
\end{flalign}
on different unbounded open domains of Heisenberg group $\mathbb H$, including the whole space, and half space of $\mathbb H$. Here $p>m_2>0$, $q>m_1>0$.
\end{abstract}

\maketitle

\section{Introduction}
Let $\mathbb H$ be Heisenberg group, which is topologically Euclidean but analytically non-Euclidean. To be precise, $\mathbb H=({\mathbb R}^{2N+1},\circ)$ is the space ${\mathbb R}^{2N+1}$ with the non-commutative law of product
$$\eta'\circ \eta=\left(x+x',y+y',\tau+\tau'+2(x\cdot y'-x'\cdot y)\right),$$
for all $\eta=(x,y,\tau),\eta'=(x',y',\tau')\in{\mathbb R}^N\times{\mathbb R}^N\times{\mathbb R}$, where $\cdot$ denotes the standard scalar product in ${\mathbb R}^N$. For more information on Heisenberg group, see Section \ref{Sec2}.

In this paper, we consider the quasilinear system of elliptic differential inequalities
\begin{equation}\label{equ}
 \begin{cases}
\Delta_{\mathbb H}u^{m_1}+|\eta|_{\mathbb H}^{\gamma_1}|v|^p\leq0,\quad {\rm in}\;\Omega,\\
\Delta_{\mathbb H}v^{m_2}+|\eta|_{\mathbb H}^{\gamma_2}|u|^q\leq0,\quad {\rm in}\;\Omega,
 \end{cases}
\end{equation}
where $p>m_2>0$ and $q>m_1>0$, $|\eta|_{\mathbb H}$ is defined as in (\ref{eta-h}),
 $\Delta_{\mathbb H}$ is the sub-Laplacian on $\mathbb H$ (see Section \ref{Sec2}),
and  $\Omega$ is an unbounded open subset of Heisenberg group $\mathbb H$ taking one of the following three forms
\begin{enumerate}
\item[(1).]{$\Omega_1:=\mathbb H$;}
\item[(2).]{$\Omega_2:=\left\{(x,y,\tau)\in\mathbb H\mid x_1>0\right\}$;}
\item[(3).]{$\Omega_3:=\left\{(x,y,\tau)\in\mathbb H\mid \tau>0\right\}.$}
\end{enumerate}

In the past several decades, more and more attentions has been attracted to the analysis and PDEs on Heisenberg group,
see \cite{BCC, Bir, CPDT, HK, GL-90, GL-92, Gre,JL88, LU, PV00, Ug99, Ugu99}.
Recall the celebrated results from Birindelli, Capuzzo  Dolcetta and Cutr$\grave{\i}$ in \cite{BCC}, they investigated
\begin{equation}\label{eq-ineq}
 \Delta_{\mathbb H}u+|\eta|_{\mathbb H}^{\gamma}u^p\leq0,\quad\mbox{in $\Omega$},
\end{equation}
and proved that
\begin{enumerate}
\item[(a).]{
if $\gamma>-2$, and
$1<p\leq\frac{Q+\gamma}{Q-2},$
then (\ref{eq-ineq}) admits no positive solutions in $\Omega_1$.}
\item[(b).]{if $\gamma>-1$, and $1<p\leq\frac{Q+\gamma}{Q-1}$, then (\ref{eq-ineq}) admits no positive solutions in $\Omega_2$.}

\item[(c).]{if $\gamma>0$, and $1<p\leq \frac{Q+\gamma}{Q}$, then (\ref{eq-ineq}) admits no positive solutions in $\Omega_3$,}
\end{enumerate}
where
\begin{equation}
Q=2N+2,
\nonumber
\end{equation}
and usually is called the homogeneous dimension of $\mathbb H$.

Later in \cite{PV00}, Pohozaev and V$\acute{\rm e}$ron removed the assumption of positiveness of solution, and stuided
\begin{equation}\label{eq-2}
 \Delta_{\mathbb H}u+|\eta|_{\mathbb H}^{\gamma}|u|^p\leq0,\quad \mbox{\rm in $\mathbb H$}.
\end{equation}
They proved that if $\gamma>-2$, and $1<p\leq\frac{Q+\gamma}{Q-2}$,
then (\ref{eq-2}) admits no locally integrable solution $u\in L_{loc}^p(\mathbb H, |\eta|_{\mathbb H}^{\gamma}d\eta)$.

In the same paper \cite{PV00}, Pohozaev and V$\acute{\rm e}$ron  also studied system of (\ref{equ}) under the special case of $m_1=m_2=1$, $\Omega=\mathbb H$, namely
\begin{equation}\label{eq-3}
 \begin{cases}
\Delta_{\mathbb H}u+|\eta|_{\mathbb H}^{\gamma_1}|v|^p\leq0,\quad {\rm in}\;\mathbb H,\\
\Delta_{\mathbb H}v+|\eta|_{\mathbb H}^{\gamma_2}|u|^q\leq0,\quad {\rm in}\;\mathbb H.
 \end{cases}
\end{equation}
They proved that (\ref{eq-3}) admits no solution $(u, v)\in L_{loc}^{q}(\mathbb H,|\eta|_{\mathbb H}^{\gamma_2}d\eta)\times L_{loc}^{p}(\mathbb H,|\eta|_{\mathbb H}^{\gamma_1}d\eta)$ provided that
$\gamma_1, \gamma_2>-2$, and
$$Q\leq2+\min\left\{\frac{2+\gamma_1}{p-1},\frac{2+\gamma_2}{q-1}\right\},
\quad\mbox{for $p>1, q>1$}.$$
Later, Hamidi and Kirane \cite{HK}  showed that (\ref{eq-3}) admits no nontrivial solution, if
$$Q\leq2+\max\left\{\frac{\gamma_1+2+p(\gamma_2+2)}{pq-1},\frac{\gamma_2+2+q(\gamma_1+2)}{pq-1}\right\}.$$
One can easily check that when $p>1, q>1$,
$$\max\left\{\frac{\gamma_1+2+p(\gamma_2+2)}{pq-1},\frac{\gamma_2+2+q(\gamma_1+2)}{pq-1}\right\}\geq
\min\left\{\frac{2+\gamma_1}{p-1},\frac{2+\gamma_2}{q-1}\right\}.$$

Motivated by the above literature, we would like to generalize the study in two respects:
The first is that we aim to remove the positive assumption of  $(u, v)$ to problem (\ref{equ}) with generalized $m_1, m_2>0$; the second is that we will study the nonexistence
results in three different domains $\Omega_1, \Omega_2, \Omega_3$.

For our convenience, throughout the paper, let us denote
\begin{equation}
\Lambda:=2+\max\left\{\frac{m_1\left[m_2(\gamma_1+2)+p(\gamma_2+2)\right]}{pq-m_1m_2},
\frac{m_2\left[m_1(\gamma_2+2)+q(\gamma_1+2)\right]}{pq-m_1m_2}\right\},
\nonumber
\end{equation}
and
\begin{equation}
\alpha:=\max\left\{\frac{p}{p-m_2}, \frac{q}{q-m_1}\right\}.
\nonumber
\end{equation}

\begin{theorem}\label{thm1}
\rm{When $\Omega=\Omega_1$. If $\gamma_1, \gamma_2>-2$, and
\begin{equation}\label{Q-1}
Q\leq\Lambda,
\end{equation}
then (\ref{equ}) admits no nontrivial solution.}
\end{theorem}

\begin{theorem}\label{thm2}
\rm{When $\Omega=\Omega_2$. If $\gamma_1, \gamma_2>-1$, and
\begin{equation}\label{Q-2}
Q\leq\Lambda-\alpha,
\end{equation}
then (\ref{equ}) admits no nontrivial solution.}
\end{theorem}

\begin{theorem}\label{thm3}
\rm{When $\Omega=\Omega_3$. If $\gamma_1, \gamma_2>0$, and
\begin{equation}\label{Q-3}
Q\leq\Lambda-2\alpha,
\end{equation}
then (\ref{equ}) admits no nontrivial solution.}
\end{theorem}

\begin{remark}
\rm{ Theorem \ref{thm2} and \ref{thm3} hold true respectively for any half-spaces
taking the forms of $$\left\{\eta\in \mathbb H: \sum_{i=1}^Na_ix_i+b_iy_i+d>0, \mbox{for
$(a, b)\in \mathbb{R}^N\times\mathbb{R}^N\setminus \{0\}, d\in \mathbb{R}$}\right\},$$
and
$$\left\{\eta\in \mathbb H: \sum_{i=1}^Na_ix_i+b_iy_i+ct+d>0, \mbox{for
$a, b\in \mathbb{R}^N, c\in\mathbb{R}\setminus \{0\}, d\in \mathbb{R}$}\right\}.$$
}
\end{remark}

The paper is organized as follows: In Section \ref{Sec2}, we prepare some preliminaries. In Section \ref{Sec3}, we give the proof of Theorem
 \ref{thm1} by applying three different test functions. Section \ref{Sec4} is devoted to the proof of Theorems \ref{thm2}-\ref{thm3}.

Throughout the paper, we denote by $c_1, c_2,\cdots, C_1, C_2\cdots$ some positive constants, which may vary from line to line. And $f\lesssim h$ means that $f\leq Ch$ for some constant $C>0$.

\section{Preliminaries}\label{Sec2}

The sub-Laplacian $\Delta_{\mathbb H}$ on $\mathbb H$ is defined, from
the vector fields
$$X_i=\partial_{x_i}+2y_i\partial_\tau,\quad
Y_i=\partial_{y_i}-2x_i\partial_\tau,\quad(i=1,\cdot\cdot\cdot N),$$
by
$$\Delta_{\mathbb H}=\sum_{i=1}^N\left(X_i^2+Y_i^2\right)=\sum_{i=1}^N\left[
\partial_{x_i}^2+\partial_{y_i}^2+4y_i\partial_{x_i\tau}^2-
4x_i\partial_{y_i\tau}^2+4(x_i^2+y_i^2)\partial_\tau^2\right].$$
For a a function $f:{\mathbb H}\rightarrow{\mathbb R}$, the horizontal gradient of $f$ is defined as
$$\nabla_{\mathbb H}f=\left(X_1f,\cdot\cdot\cdot,X_Nf,Y_1f,\cdot\cdot\cdot,Y_Nf\right)
=\sum_{i=1}^N\left[\left(X_if\right)X_i+\left
(Y_if\right)Y_i\right].$$
Let us define the norm of $\eta\in\mathbb H$ by
\begin{equation}\label{eta-h}
\left|\eta\right|_{\mathbb H}=\left(\left(|x|^2+|y|^2\right)^2+\tau^2\right)^{\frac{1}
{4}},
\end{equation}
which is homogeneous of degree $1$ with respect to the dilations $\delta_\lambda:(x,y,\tau)\mapsto(\lambda x,\lambda y,\lambda^2\tau)$
for $\lambda>0$.
And Heisenberg distance between $\eta$ and $h$ on $\mathbb H$ is defined by
$$d_{\mathbb H}(\eta,h)=\left|\eta^{-1}\circ h\right|_{\mathbb H},$$
where $\eta^{-1}=-\eta$.

Let us define the Heisenberg ball of radius $R$ centered at $\eta$ be the set
$$B_{\mathbb H}(\eta,R)=\left\{h\in{\mathbb H}:d_{\mathbb H}(\eta,h)<R\right\},$$
it follows that
$$\left|B_{\mathbb H}(\eta,R)\right|=\left|B_{\mathbb H}(0,R)\right|=\left|B_{\mathbb H}(0,1)\right|R^Q,$$
where $|B_{\mathbb H}(0,1)|$ is the volume of the unit Heisenberg ball under Haar measure, which is equivalent to $(2N+1)$-dimensional Lebesgue measure of $\mathbb R^{2N+1}$, and $Q=2N+2$ is called the homogeneous dimension of $\mathbb H$. For more details concerning the Heisenberg group, one can refer to books as \cite{CPDT,IV}, survey papers as \cite{Fol,GL-90,Gre} and the references therein.

Define
$$W_{loc}^{1,p}(\Omega)=\left\{u:\Omega\rightarrow\mathbb R\mid u, \nabla_{\mathbb H}u\in L_{loc}^p(\Omega)\right\},$$
and let $W_c^{1,p}(\Omega)$ be the subspace of $W_{loc}^{1,p}(\Omega)$ of functions with compact support.

\begin{definition}
\rm A pair $(u,v)$ is called a weak solution to system (\ref{equ}) if
$(u,v) \in W_{loc}^{1,q}\left(\Omega\right)\times W_{loc}^{1,p}\left( \Omega\right)$, and
the following inequalities
\begin{align}\label{sol-1}
\int_{\Omega}|\eta|_{\mathbb H}^{\gamma_1}|v|^p\psi d\eta\leq-\int_{\Omega}\psi\Delta_{\mathbb H}u^{m_1} d\eta,
\end{align}
and
\begin{align}\label{sol-2}
\int_{\Omega}|\eta|_{\mathbb H}^{\gamma_2}|u|^q\psi d\eta\leq-\int_{\Omega}\psi\Delta_{\mathbb H}v^{m_2} d\eta,
\end{align}
are valid for any  $0\leq \psi\in W_{c}^{1,q}\left(\Omega\right)\cap W_{c}^{1,p}\left(\Omega\right)$.
\end{definition}

Let
\begin{equation}
D_R:=B_{\mathbb H}(0,2kR),\quad\mbox{for $k\geq1$.}
\nonumber
\end{equation}
Define
\begin{eqnarray}
D_i:=\Omega_i\cap D_R,\quad\mbox{for $i=1, 2, 3$,}
\nonumber
\end{eqnarray}
and
\begin{equation}
f_1:=1,\quad f_2:=x_1^\alpha,\quad f_3:=\tau^\alpha.
\nonumber
\end{equation}

\begin{lemma}
\rm Assume that $(u,v)$ is a solution to (\ref{equ}).
 Let $\varphi\in W_{c}^{1,q}\left(\Omega\right)\cap W_{c}^{1,p}\left(\Omega\right)$ satisfying that $0\leq\varphi\leq1$ and ${\rm supp}\varphi\subset D_R$, we have for $i=1,2,3$, and $j=1,2$,
 \begin{equation}\label{est-IJ}
 \begin{cases}
  I_i^{1-\frac{m_1m_2}{pq}}\lesssim\left(K_{i,1}+ L_{i,1}\right)^{\frac{m_1}{q}}\left(K_{i,2}+  L_{i,2}\right),\\
 J_i^{1-\frac{m_1m_2}{pq}}\lesssim \left(K_{i,1}+ L_{i,1}\right)\left(K_{i,2}+ L_{i,2}\right)
^{\frac{m_2}{p}},
 \end{cases}
\end{equation}
 where
 \begin{flalign}
&& &
I_i:=\int_{D_i}|\eta|_{\mathbb H}^{\gamma_1}|v|^pf_i\varphi^bd\eta,
\quad J_i:=\int_{D_i}|\eta|_{\mathbb H}^{\gamma_2}|u|^qf_i\varphi^bd\eta,
&\nonumber\\&& &
K_{i,j}:=\left(\int_{D_i}|\eta|_{\mathbb H}^{(1-\lambda_j)\gamma_j}f_i\left|\Delta_{\mathbb
H}\varphi\right|^{\lambda_j} d\eta\right)^{\frac{1}{\lambda_j}},
&\nonumber\\&& &
L_{i,j}:=\left(\int_{D_i}|\eta|_{\mathbb H}^{(1-\lambda_j)\gamma_j}f_i^{1-\lambda_j}
\left|\nabla_{\mathbb H}f_i\right|^{\lambda_j}
\left|\nabla_{\mathbb H}\varphi\right|^{\lambda_j}
d\eta\right)^{\frac{1}{\lambda_j}},
&\nonumber\\&& &
\lambda_1:=\frac{p}{p-m_2},\quad\lambda_2:=\frac{q}{q-m_1}.
&\nonumber
\end{flalign}
\end{lemma}

\begin{proof}\rm

Let $$\psi_i=f_i\varphi^b,\quad (i=1,2,3),$$
where $b>1$ is a large enough constant. Note that on $\partial D_i$,
$$\psi_i=0,$$ and $$\nabla_{\mathbb H}\psi_i=bf_i\varphi^{b-1}
\nabla_{\mathbb H}\varphi+\varphi^b\nabla_{\mathbb H}f_i=0,$$
with $$\nabla_{\mathbb H}f_1=0,\quad\nabla_{\mathbb H}f_2=\alpha x_1^{\alpha-1}\nabla_{\mathbb H}x_1,\quad\nabla_{\mathbb H}f_3=\alpha \tau^{\alpha-1}\nabla_{\mathbb H}\tau,$$and $$\nabla_{\mathbb H}x_1=(1,0,\cdot\cdot\cdot,0),\quad\nabla_{\mathbb H}\tau=2(y,x).$$

Substituting $\psi=\psi_i=f_i\varphi^b$ into (\ref{sol-1}) and (\ref{sol-2}), we obtain
\begin{align}\label{es-1}
\int_{D_i}|\eta|_{\mathbb H}^{\gamma_1}|v|^pf_i\varphi^b d\eta\leq&
-\int_{D_i}u^{m_1}\Delta_{\mathbb H}\left(f_i\varphi^b\right) d\eta,
\end{align}
and
\begin{align}\label{es-2}
\int_{D_i}|\eta|_{\mathbb H}^{\gamma_2}|u|^qf_i\varphi^b d\eta\leq&
-\int_{D_i}v^{m_2}\Delta_{\mathbb H}\left(f_i\varphi^b\right)d\eta.
\end{align}
 Note also that
$$\Delta_{\mathbb H}f_1=0,\quad
\Delta_{\mathbb H}f_2=\alpha(\alpha-1)x_1^{\alpha-2}\geq0,\quad
\Delta_{\mathbb H}f_3=4\alpha(\alpha-1)(x^2+y^2)\tau^{\alpha-2}\geq0,$$
and
$$\Delta_{\mathbb
H}\varphi^b=b(b-1)\varphi^{b-2}\left|\nabla_{\mathbb
H}\varphi\right|^2+b\varphi^{b-1}\Delta_{\mathbb
H}\varphi\geq
b\varphi^{b-1}\Delta_{\mathbb H}\varphi.$$
It follows that
\begin{align}\label{es-3}
\Delta_{\mathbb H}\left(f_i\varphi^b\right)
&= f_i\Delta_{\mathbb
H}\varphi^b+2\nabla_{\mathbb
H}f_i\cdot\nabla_{\mathbb H}\varphi^b+\varphi^b\Delta_{\mathbb H}f_i
\nonumber\\
&\geq bf_i\varphi^{b-1}
\Delta_{\mathbb H}\varphi+2b
\varphi^{b-1}\nabla_{\mathbb
H}f_i\cdot
\nabla_{\mathbb H}\varphi.
\end{align}
Combining (\ref{es-1}) with (\ref{es-3}), we obtain
\begin{align}
\int_{D_i}|\eta|_{\mathbb H}^{\gamma_1}|v|^pf_i\varphi^b d\eta\leq\;&b
\int_{D_i}|u|^{m_1}f_i\varphi^{b-1}
\left|\Delta_{\mathbb H}\varphi\right| d\eta\nonumber\\  &
+2b\int_{D_i}|u|^{m_1}\varphi^{b-1}\left|\nabla_{\mathbb
H}f_i\right|\left|
\nabla_{\mathbb H}\varphi\right|d\eta.\nonumber
\end{align}
Applying H$\ddot{\rm o}$lder's inequality, we arrive
\begin{align*}
&\int_{D_i}|\eta|_{\mathbb H}^{\gamma_1}|v|^pf_i\varphi^b d\eta
\nonumber\\
\lesssim&\left(\int_{D_i\cap{\rm supp \left(\nabla_{\mathbb H}\varphi\right)}}|\eta|_{\mathbb H}^{\gamma_2}|u|^qf_i\varphi^b d\eta\right)^{\frac{m_1}{q}}
\left\{\left(\int_{D_i}|\eta|_{\mathbb H}^{-\frac{m_1\gamma_2}{q-m_1}}f_i\varphi^{b-\frac{q}{q-m_1}}\left|\Delta_{\mathbb H}\varphi\right|^{\frac{q}{q-m_1}}d\eta\right)^{\frac{q-m_1}{q}}\right.
\nonumber\\ &\left.
+
\left(\int_{D_i}|\eta|_{\mathbb H}^{-\frac{m_1\gamma_2}{q-m_1}}f_i^{-\frac{m_1}{q-m_1}}\varphi^{b-\frac{q}{q-m_1}
}\left|\nabla_{\mathbb H}f_i\right|^{\frac{q}{q-m_1}}\left|\nabla_{\mathbb H}\varphi\right|^{\frac{q}{q-m_1}}d\eta\right)^{\frac{q-m_1}{q}}\right\}.
\nonumber
\end{align*}
Since $0\leq\varphi\leq1$, we can chose $b$ large enough such that
\begin{align}\label{es-4}
&\int_{D_i}|\eta|_{\mathbb H}^{\gamma_1}|v|^pf_i\varphi^b d\eta
\nonumber\\
\lesssim&\left(\int_{D_i\cap{\rm supp \left(\nabla_{\mathbb H}\varphi\right)}}|\eta|_{\mathbb H}^{\gamma_2}|u|^qf_i\varphi^b d\eta\right)^{\frac{m_1}{q}}
\left\{\left(\int_{D_i}|\eta|_{\mathbb H}^{-\frac{m_1\gamma_2}{q-m_1}}f_i\left|\Delta_{\mathbb H}\varphi\right|^{\frac{q}{q-m_1}}d\eta\right)^{\frac{q-m_1}{q}}\right.
\nonumber\\ &\left.
+
\left(\int_{D_i}|\eta|_{\mathbb H}^{-\frac{m_1\gamma_2}{q-m_1}}f_i^{-\frac{m_1}{q-m_1}}\left|\nabla_{\mathbb H}f_i\right|^{\frac{q}{q-m_1}}\left|\nabla_{\mathbb H}\varphi\right|^{\frac{q}{q-m_1}}d\eta\right)^{\frac{q-m_1}{q}}\right\}.
\end{align}
Similarly,
\begin{align}\label{es-5}
&\int_{D_i}|\eta|_{\mathbb H}^{\gamma_2}|u|^qf_i\varphi^b d\eta
\nonumber\\
\lesssim&\left(\int_{D_i\cap{\rm supp \left(\nabla_{\mathbb H}\varphi\right)}}|\eta|_{\mathbb H}^{\gamma_1}|v|^pf_i\varphi^b d\eta\right)^{\frac{m_2}{p}}
\left\{\left(\int_{D_i}|\eta|_{\mathbb H}^{-\frac{m_2\gamma_1}{p-m_2}}f_i\left|\Delta_{\mathbb H}\varphi\right|^{\frac{p}{p-m_2}}d\eta\right)^{\frac{p-m_2}{p}}\right.
\nonumber\\ &\left.
+
\left(\int_{D_i}|\eta|_{\mathbb H}^{-\frac{m_2\gamma_1}{p-m_2}}f_i^{-\frac{m_2}{p-m_2}}\left|\nabla_{\mathbb H}f_i\right|^{\frac{p}{p-m_2}}\left|\nabla_{\mathbb H}\varphi\right|^{\frac{p}{p-m_2}}d\eta\right)^{\frac{p-m_2}{p}}\right\}.
\end{align}
%For simplicity, we put
%\begin{flalign}
%&& &
%I_i:=\int_{D_i}v^pf_i\varphi^bdg,
%&\nonumber\\&& &
%J_i:=\int_{D_i}u^qf_i\varphi^bdg,
%&\nonumber\\&& &
%K_i(\lambda):=\left(\int_{D_i}f_i\left|\Delta_{\mathbb
%H}\varphi\right|^\lambda dg\right)^{\frac{1}{\lambda}},&\nonumber\\&& &
%L_i(\lambda):=\left(\int_{D_i}f_i^{1-\lambda}
%\left|\nabla_{\mathbb H}f_i\right|^\lambda
%\left|\nabla_{\mathbb H}\varphi\right|^\lambda
%dg\right)^{\frac{1}{\lambda}},
%&\nonumber\\&& &
%\lambda_1:=\frac{q}{q-m_1},\quad\lambda_2:=\frac{p}{p-m_2}.
%&\nonumber
%\end{flalign}
Combining (\ref{es-4}) with (\ref{es-5}), we obtain
\begin{equation*}
 \begin{cases}
 I_i^{1-\frac{m_1m_2}{pq}}\lesssim\left(K_{i,1}+ L_{i,1}\right)^{\frac{m_1}{q}}\left(K_{i,2}+  L_{i,2}\right),\\
 J_i^{1-\frac{m_1m_2}{pq}}\lesssim \left(K_{i,1}+ L_{i,1}\right)\left(K_{i,2}+ L_{i,2}\right)
^{\frac{m_2}{p}},
 \end{cases}
\end{equation*}
which completes the proof.
\end{proof}

\section{Proof of Theorem \ref{thm1}}\label{Sec3}

\rm{\begin{proof}[Proof of Theorem \ref{thm1}]
When $\Omega=\Omega_1$, we have for $j=1,2$,
$$K_{1,j}=\left(\int_{D_1}|\eta|_{\mathbb H}^{(1-\lambda_j)\gamma_j}\left|\Delta_{\mathbb
H}\varphi\right|^{\lambda_j} d\eta\right)^{\frac{1}{\lambda_j}},$$
and
$$L_{1,j}=0.$$
Then (\ref{est-IJ}) becomes
\begin{equation}\label{1-1}
 \begin{cases}
 I_1^{1-\frac{m_1m_2}{pq}}\lesssim \left(K_{1,1}\right)^{\frac{m_1}{q}}K_{1,2},\\
 J_1^{1-\frac{m_1m_2}{pq}}\lesssim K_{1,1}\left(K_{1,2}\right)
^{\frac{m_2}{p}}.
 \end{cases}
\end{equation}
In what follows, we would like to use (\ref{1-1}) with three different types of test functions to prove $(u,v)\equiv(0,0)$ in $\mathbb H$, respectively.

$\bullet$ One is
\begin{equation}\label{1-2}
\varphi(\eta)=\varphi(x,y,\tau):=\phi\left(\frac{|x|^4+|y|^4+\tau^2}{R^4}
\right),
\end{equation}
where $\phi\in C^\infty[0,\infty)$ is a nonnegative function satisfying
\begin{equation}
\phi(t)=1\;{\rm on}\;[0,1];\quad \phi(t)=0\;{\rm on}\;[2,\infty);\quad
|\phi'|\leq C<\infty.\nonumber
\end{equation}
Set $r=\frac{|x|^4+|y|^4+\tau^2}{R^4}$. Note that ${\rm supp}\left(\nabla_{\mathbb H}\varphi\right)$ is a subset of
\begin{equation}
\Sigma_R=\left\{\eta=(x,y,\tau)\in\mathbb H\mid R^4\leq |x|^4+|y|^4+\tau^2\leq2R^4\right\}.\nonumber
\end{equation}
Direct calculation yields that
\begin{align}
\Delta_{\mathbb H}\varphi=\;&\sum_{i=1}^N\left[
\partial_{x_i}^2\varphi+\partial_{y_i}^2\varphi+4y_i\partial_{x_i\tau}
^2\varphi-4x_i\partial_{y_i\tau}^2\varphi+4(x_i^2+y_i^2)
\partial_\tau^2\varphi\right]
\nonumber\\
=\;&16R^{-8}\phi''(r)\left[|x|^6+|y|^6+\tau^2(|x|^2+|y|^2)
+2\tau x\cdot y(|x|^2-|y|^2)\right]
\nonumber\\&
+4(4+N)R^{-4}\phi'(r)(|x|^2+|y|^2)
.\nonumber
\end{align}
Thus,
\begin{equation}\label{1-3}
\left|\Delta_{\mathbb H}\varphi\right|\lesssim R^{-2}.
\end{equation}
It follows that
\begin{align}\label{1-4}
K_{1,j}&=\left(\int_{\Sigma_R}|\eta|_{\mathbb H}^{(1-\lambda_j)\gamma_j}\left|\Delta_{\mathbb
H}\varphi\right|^{\lambda_j} d\eta\right)^{\frac{1}{\lambda_j}}
\nonumber\\ &\lesssim
\left(R^{(1-\lambda_j)\gamma_j}R^{-2\lambda_j}\int_{\Sigma_R}d\eta\right)
^{\frac{1}{\lambda_j}}
\nonumber\\ &\lesssim R^{\frac{(1-\lambda_j)\gamma_j+Q}{\lambda_j}-2}.
\end{align}
Inserting (\ref{1-4}) into (\ref{1-1}), we compute
\begin{equation}\label{1-5}
 \begin{cases}
I_1^{1-\frac{m_1m_2}{pq}}\lesssim R^{\sigma_{I_1}},\\
J_1^{1-\frac{m_1m_2}{pq}}\lesssim R^{\sigma_{J_1}},
\end{cases}
\end{equation}
where
\begin{align}\label{I}
\sigma_{I_1}:=\frac{Q\left(pq-m_1m_2\right)}{pq}
-\frac{m_1(p\gamma_2+m_2\gamma_1)}{pq}-\frac{2(m_1+q)}{q},\\ \label{J}
\sigma_{J_1}:=\frac{Q\left(pq-m_1m_2\right)}{pq}
-\frac{m_2(q\gamma_1+m_1\gamma_2)}{pq}-\frac{2(m_2+p)}{p}.
\end{align}

Note that $\sigma_{I_1}\leq0$ or $\sigma_{J_1}\leq0$ if and only if (\ref{Q-1}) holds.
In the case $\sigma_{I_1}\leq0$, the integral $I_1$, increasing in $R$, is bounded uniformly with respect to $R$. Applying the monotone convergence theorem, we conclude that $|\eta|_{\mathbb H}^{\gamma_1}|v|^p$ is in $L^1\left(\mathbb H\right)$. Note that instead of the first inequality of (\ref{1-5}) we have, more precisely,
\begin{equation}
I_1\lesssim \left(\int_{\Sigma_R}|\eta|_{\mathbb H}^{\gamma_1}|v|^p\varphi^b d\eta\right)^{\frac{m_1m_2}{pq}}R^{\sigma_{I_1}}\lesssim \left(\int_{\Sigma_R}|\eta|_{\mathbb H}^{\gamma_1}|v|^p\varphi^b d\eta\right)^{\frac{m_1m_2}{pq}}.
\nonumber
\end{equation}
Finally, using the dominated convergence theorem, we obtain
$$\lim_{R\rightarrow+\infty}\int_{\Sigma_R}|\eta|_{\mathbb H}^{\gamma_1}|v|^p\varphi^b d\eta=0.$$
Therefore,
$$\lim_{R\rightarrow+\infty}I_1=0,$$
which implies that $v\equiv0$ in $\mathbb H$ and thus $u\equiv0$ in $\mathbb H$ via (\ref{es-5}).
The proof in the case $\sigma_{J_1}\leq0$ is analogous.

$\bullet$ Another is
$$
 \varphi(\eta):=\omega(\eta)\xi_k(\eta),\quad{\rm for\;\, fixed}\;\,  k\in\mathbb{N},$$
where\begin{flalign}
 && &
 \omega(\eta)=
 \begin{cases}
 1,\qquad\quad\;\, \rho<R,\\
 \left(\frac{\rho}{R}\right)^{-\delta},\quad \rho\geq R,
 \end{cases}
 &\nonumber
\end{flalign}
and
\begin{flalign}
 && &
 \xi_k(\eta)=
 \begin{cases}
 1,\qquad\quad\;\; 0\leq\rho\leq kR,\\
 2-\frac{\rho}{kR},\quad kR\leq\rho\leq2kR,\\
 0,\qquad\quad\;\; \rho\geq 2kR,
 \end{cases}
 &\nonumber
\end{flalign}
with $\rho:=\left|\eta\right|_{\mathbb H}$, and
\begin{equation}\label{1-6}
\delta>\frac{\Lambda+(1-\lambda_j)\gamma_j}{\lambda_j}-2,\quad{\rm for}\;\, j=1,2.
\end{equation}
Note that ${\rm supp}\left(\nabla_{\mathbb H}\varphi\right)$ is a subset of
\begin{equation}
\Sigma_R'=\left\{\eta=(x,y,\tau)\in\mathbb H\mid \eta\in B_{\mathbb H}(0,2kR)\backslash B_{\mathbb H}(0,R)\right\}.\nonumber
\end{equation}
According to \cite{BCC}, we have
$$\Delta_{\mathbb H}\rho=\frac{Q-1}{\rho}\Psi(\eta),
$$
where the function $\Psi$ is defined by
$$\Psi(\eta)=\frac{|x|^2+|y|^2}{\rho^2}=
\left|\nabla_{\mathbb H}\rho\right|^2,\quad {\rm for}\; \eta\neq0.$$
Note that $0\leq\Psi\leq1$, it is not difficult to check that
\begin{equation}\label{1-7}
 \left|\nabla_{\mathbb H}\omega\right|=\delta R^\delta\rho^{-(\delta+1)}
 \left|\nabla_{\mathbb H}\rho\right|
\lesssim  R^\delta\rho^{-(\delta+1)},
\end{equation}
\begin{equation}\label{1-8}
 \left|\Delta_{\mathbb H}\omega\right|=\delta R^\delta\rho^{-(\delta+2)}\left|\rho
\Delta_{\mathbb H}\rho-(\delta+1)
 \left|\nabla_{\mathbb H}\rho\right|^2\right|
\lesssim  R^\delta\rho^{-(\delta+2)},
\end{equation}
and
\begin{equation}\label{1-9}
\left|\nabla_{\mathbb H}\xi_k\right|=(kR)^{-1}\left|\nabla_{\mathbb H}\rho\right|\lesssim (kR)^{-1},
\end{equation}
\begin{equation}\label{1-10}
\left|\Delta_{\mathbb H}\xi_k\right|=(kR)^{-1}\left|\Delta_{\mathbb H}\rho\right|\lesssim (kR\rho)^{-1}.
\end{equation}
Note also that $\varphi(\eta)\uparrow\omega(\eta)$ as $k\rightarrow\infty$, and for every $\lambda\geq1$,
\begin{flalign}
 && &
 \left|\Delta_{\mathbb H}\varphi\right|^\lambda\lesssim \omega^\lambda\left|\Delta_{\mathbb H}\xi_k\right|^\lambda
+\xi_k^\lambda\left|\Delta_{\mathbb H}\omega\right|^\lambda+\left|\nabla_{\mathbb H}\omega\right|^\lambda\left|\nabla_{\mathbb H}\xi_k\right|^\lambda.
 &\nonumber
\end{flalign}
Then
\begin{align}\label{1-11}
\left(K_{1,j}\right)^{\lambda_j}=\;&\int_{\mathbb H}|\eta|_{\mathbb H}^{(1-\lambda_j)\gamma_j}\left|\Delta_{\mathbb H}\varphi\right|^{\lambda_j} d\eta
\nonumber\\ \lesssim\;&
\int_{B_{\mathbb H}(0,2kR)\backslash B_{\mathbb H}(0,kR)}
|\eta|_{\mathbb H}^{(1-\lambda_j)\gamma_j}\omega^{\lambda_j}\left|\Delta_{\mathbb H}\xi_k\right|^{\lambda_j} d\eta
\nonumber\\&
+\int_{B_{\mathbb H}(0,2kR)\backslash B_{\mathbb H}(0,R)}
|\eta|_{\mathbb H}^{(1-\lambda_j)\gamma_j}\xi_k^{\lambda_j}\left|\Delta_{\mathbb H}\omega\right|^{\lambda_j} d\eta
\nonumber\\&
+\int_{B_{\mathbb H}(0,2kR)\backslash B_{\mathbb H}(0,kR)}|\eta|_{\mathbb H}^{(1-\lambda_j)\gamma_j}\left|\nabla_{\mathbb H}\omega\right|^{\lambda_j}\left|\nabla_{\mathbb H}\xi_k\right|^{\lambda_j}
d\eta
\nonumber\\
=\;&M_1+M_2+M_3.
\end{align}
By (\ref{1-10}), we have
\begin{align}\label{1-12}
M_1&\lesssim(kR)^{-\lambda_j}\int_{B_{\mathbb H}(0,2kR)\backslash B_{\mathbb H}(0,kR)}\omega^{\lambda_j}\rho^{(1-\lambda_j)\gamma_j}\rho^{-\lambda_j} d\eta
\nonumber\\&\lesssim
(kR)^{-\lambda_j}\left(\sup_{B_{\mathbb H}(0,2kR)\backslash B_{\mathbb H}(0,kR)}\omega^{\lambda_j}\right)\int_{B_{\mathbb H}(0,2kR)\backslash B_{\mathbb H}(0,kR)}\rho^{(1-\lambda_j)\gamma_j-\lambda_j} d\eta
\nonumber\\&\lesssim(kR)^{-\lambda_j}
\left(\frac{kR}{R}\right)^{-\delta\lambda_j}\int_{kR}^{2kR}\rho^
{(1-\lambda_j)\gamma_j-\lambda_j+Q-1}d\rho
\nonumber\\&\lesssim
k^{(1-\lambda_j)\gamma_j-(\delta+2)\lambda_j+Q}
R^{(1-\lambda_j)\gamma_j-2\lambda_j+Q}.
\end{align}
By (\ref{1-8}), we have
\begin{align}\label{1-13}
M_2&\lesssim  R^{\delta\lambda_j}
\int_{B_{\mathbb H}(0,2kR)\backslash B_{\mathbb H}(0,R)}\rho^{(1-\lambda_j)\gamma_j}\rho^{-(\delta+2)\lambda_j}d\eta
\nonumber\\&\lesssim R^{\delta\lambda_j}
\int_{R}^{2kR}\rho^{(1-\lambda_j)\gamma_j-(\delta+2)\lambda_j+Q-1}d\rho
\nonumber\\&\lesssim\frac{(2k)^{(1-\lambda_j)\gamma_j-(\delta+2)\lambda_j+Q}-1}
{(1-\lambda_j)\gamma_j-(\delta+2)\lambda_j+Q} R^{(1-\lambda_j)\gamma_j-2\lambda_j+Q}.
\end{align}
By (\ref{1-7}) and (\ref{1-9}), we have
\begin{align}\label{1-14}
M_3&\lesssim(kR)^{-\lambda_j}R^{\delta\lambda_j}\int_{B_{\mathbb H}(0,2kR)\backslash B_{\mathbb H}(0,kR)}\rho^{(1-\lambda_j)\gamma_j}\rho^{-(\delta+1)\lambda_j}d\eta
\nonumber\\&\lesssim
(kR)^{-\lambda_j}R^{\delta\lambda_j}\int_{kR}^{2kR}\rho^{(1-\lambda_j)
\gamma_j-(\delta+1)\lambda_j+Q-1}d\rho
\nonumber\\&\lesssim
k^{(1-\lambda_j)\gamma_j-(\delta+2)\lambda_j+Q}
R^{(1-\lambda_j)\gamma_j-2\lambda_j+Q}.
\end{align}
A combination of (\ref{1-11}), (\ref{1-12}), (\ref{1-13}) and (\ref{1-14}) yields that
\begin{align}
K_{1,j}\lesssim \left(k^{(1-\lambda_j)\gamma_j-(\delta+2)\lambda_j+Q}
+ \frac{(2k)^{(1-\lambda_j)\gamma_j-(\delta+2)\lambda_j+Q}-1}
{(1-\lambda_j)\gamma_j-(\delta+2)\lambda_j+Q} \right)^{\frac{1}{\lambda_j}}R^{\frac{(1-\lambda_j)\gamma_j+Q}{\lambda_j}-2}.\nonumber
\end{align}
It follows from (\ref{Q-1}) and (\ref{1-6}) that $(1-\lambda_j)\gamma_j-(\delta+2)\lambda_j+Q<0$.
Thus, upon taking $k\rightarrow\infty$, we obtain
\begin{equation}\label{1-15}
K_{1,j}\lesssim R^{\frac{(1-\lambda_j)\gamma_j+Q}{\lambda_j}-2}.
\end{equation}
Inserting (\ref{1-15}) into (\ref{1-1}), we derive (\ref{1-5}) once again.
By substituting $\Sigma_R'$ for $\Sigma_R$, and arguing as we did as above, we can get the desired result.

$\bullet$ The third is
\begin{equation}
\varphi(\eta):=\frac{1}{n}\sum_{k=n+1}^{2n}\varphi_k(\eta),\quad{\rm for\;\, fixed}\;\,  n\in\mathbb{N},\nonumber
\end{equation}
where $\left\{\varphi_k\right\}_{k\in\mathbb{N}}$ is a sequence satisfying that each $\varphi_k$ is a Lipschitz function such that ${\rm supp}\left(\varphi_k\right)\subset B_{\mathbb H}(0,2^k)$, $\varphi_k=1$ in a neighborhood of $B_{\mathbb H}(0,2^{k-1})$, and
\begin{align}
\left|\Delta_{\mathbb H}\varphi_k\right|
 \begin{cases}
\lesssim \frac{1}{2^{2(k-1)}},\quad \eta\in B_{\mathbb H}(0,2^k)\backslash B_{\mathbb H}(0,2^{k-1}),\\
=0,\quad {\rm otherwise}.
 \end{cases}
 &\nonumber
\end{align}
Note that $\varphi=1$ on $B_{\mathbb H}(0,2^n)$, $\varphi=0$ outside $B_{\mathbb H}(0,2^{2n})$, and $0\leq\varphi\leq1$ on $\mathbb H$. Note also that for distinct $k$, ${\rm supp}\left(\Delta_{\mathbb H}\varphi_k\right)$ are disjoint with each other. Then we have for any $\lambda>0$,
$$\left|\Delta_{\mathbb H}\varphi\right|^\lambda=n^{-\lambda}
 \sum_{k=n+1}^{2n}\left|\Delta_{\mathbb H}\varphi_k\right|^\lambda.$$
It follows that
\begin{align}
\left(K_{1,j}\right)^{\lambda_j}&=n^{-\lambda_j}\int_{\mathbb H}
|\eta|_{\mathbb H}^{(1-\lambda_j)\gamma_j}\sum_{k=n+1}^{2n}\left|\Delta_{\mathbb H}\varphi_k\right|^{\lambda_j} d\eta
\nonumber\\&=
n^{-\lambda_j}\sum_{k=n+1}^{2n}
\int_{B_{\mathbb H}(0,2^k)\backslash B_{\mathbb H}(0,2^{k-1})}|\eta|_{\mathbb H}^{(1-\lambda_j)\gamma_j}\left|\Delta_{\mathbb H}\varphi_k\right|^{\lambda_j}  d\eta
\nonumber\\& \lesssim n^{-\lambda_j} \sum_{k=n+1}^{2n}2^{k(1-\lambda_j)\gamma_j}2^{-2k\lambda_j}2^{kQ}
\nonumber\\& \lesssim n^{-\lambda_j+1}2^{\left[(1-\lambda_j)\gamma_j-
2\lambda_j+Q\right]n}.\nonumber
\end{align}
Thus,
\begin{equation}\label{1-16}
K_{1,j}\lesssim n^{-1+\frac{1}{\lambda_j}}2^{\left(\frac{(1-\lambda_j)\gamma_j+Q}{\lambda_j}
-2\right)n}.
\end{equation}
Inserting (\ref{1-16}) into (\ref{1-1}), we compute
\begin{equation}\label{1-17}
 \begin{cases}
I_1^{1-\frac{m_1m_2}{pq}}\lesssim n^{-\frac{m_1(p+m_2)}{pq}}2^{n\sigma_{I_1}},\\
J_1^{1-\frac{m_1m_2}{pq}}\lesssim n^{-\frac{m_2(q+m_1)}{pq}}2^{n\sigma_{J_1}},
 \end{cases}
\end{equation}
where $\sigma_{I_1}$ and $\sigma_{J_1}$ are defined as in (\ref{I}) and (\ref{J}).

Note that $\sigma_{I_1}\leq0$ or $\sigma_{J_1}\leq0$ if and only if (\ref{Q-1}) holds.
In the case $\sigma_{I_1}\leq0$, by taking $n\rightarrow\infty$, we conclude from the first inequality of (\ref{1-17}) that
\begin{equation}
\lim_{n\rightarrow+\infty}I_1=0,\nonumber
\end{equation}
which implies that $v\equiv0$ in $\mathbb H$ and thus $u\equiv0$ in $\mathbb H$ via (\ref{es-5}).
The proof in the case $\sigma_{J_1}\leq0$ is analogous.
\end{proof}}

\section{Proof of Theorems \ref{thm2}-\ref{thm3}}\label{Sec4}

\begin{proof}[Proof of Theorem \ref{thm2}]
When $\Omega=\Omega_2$, we have for $j=1,2$,
$$K_{2,j}=\left(\int_{D_2}|\eta|_{\mathbb H}^{(1-\lambda_j)\gamma_j}x_1^\alpha\left|\Delta_{\mathbb
H}\varphi\right|^{\lambda_j} d\eta\right)^{\frac{1}{\lambda_j}},$$
and
$$L_{2,j}=\alpha\left(\int_{D_2}|\eta|_{\mathbb H}^{(1-\lambda_j)\gamma_j}x_1^{\alpha-\lambda_j}
\left|\nabla_{\mathbb H}x_1\right|^{\lambda_j}
\left|\nabla_{\mathbb H}\varphi\right|^{\lambda_j}
d\eta\right)^{\frac{1}{\lambda_j}}.$$
Then (\ref{est-IJ}) reads as
\begin{equation}\label{2-1}
 \begin{cases}
I_2^{1-\frac{m_1m_2}{pq}}\lesssim\left(K_{2,1}+ L_{2,1}\right)^{\frac{m_1}{q}}\left(K_{2,2}+  L_{2,2}\right),\\
J_2^{1-\frac{m_1m_2}{pq}}\lesssim \left(K_{2,1}+ L_{2,1}\right)\left(K_{2,2}+ L_{2,2}\right)
^{\frac{m_2}{p}}.
 \end{cases}
\end{equation}
Let us use the first test function $\varphi(\eta)$ defined by (\ref{1-2}). In view of (\ref{1-3}), we obtain
\begin{equation}\label{2-2}
K_{2,j}\lesssim \left(R^{(1-\lambda_j)\gamma_j}R^\alpha R^{-2\lambda_j}\int_{\Sigma_R}d\eta\right)^{\frac{1}{\lambda_j}}
\lesssim R^{\frac{{(1-\lambda_j)\gamma_j}+\alpha+Q}{\lambda_j}-2}.
\end{equation}
Also, we calculate
\begin{align}\label{2-3}
\left|\nabla_{\mathbb
H}\varphi\right|&=\left(\sum_{i=1}^N\left|
X_i\varphi\right|^2+\left|Y_i\varphi\right|^2\right)^{\frac{1}{2}}
\nonumber\\ &
=4R^{-4}\left|\phi'(r)\right|\left[|x|^6+|y|^6+\tau^2(|x|^2+|y|^2)
+2\tau x\cdot y(|x|^2-|y|^2)\right]^{\frac{1}{2}}
\nonumber\\
 &\lesssim R^{-1}.
\end{align}
Thus,
\begin{equation}\label{2-4}
L_{2,j}\lesssim \left(R^{(1-\lambda_j)\gamma_j}R^{\alpha-\lambda_j} R^{-\lambda_j}\int_{\Sigma_R}d\eta\right)^{\frac{1}{\lambda_j}}
\lesssim R^{\frac{(1-\lambda_j)\gamma_j+\alpha+Q}{\lambda_j}-2}.
\end{equation}
Inserting (\ref{2-2}) and (\ref{2-4}) into (\ref{2-1}), we compute
\begin{equation}\label{2-5}
 \begin{cases}
I_2^{1-\frac{m_1m_2}{pq}}\lesssim R^{\sigma_{I_2}},\\
J_2^{1-\frac{m_1m_2}{pq}}\lesssim R^{\sigma_{J_2}},
 \end{cases}
\end{equation}
where
\begin{align}
\sigma_{I_2}:=\frac{(\alpha+Q)\left(pq-m_1m_2\right)}{pq}
-\frac{m_1(p\gamma_2+m_2\gamma_1)}{pq}-\frac{2(m_1+q)}{q},\nonumber\\
\sigma_{J_2}:=\frac{(\alpha+Q)\left(pq-m_1m_2\right)}{pq}
-\frac{m_2(q\gamma_1+m_1\gamma_2)}{pq}-\frac{2(m_2+p)}{p}.\nonumber
\end{align}

Note that $\sigma_{I_2}\leq0$ or $\sigma_{J_2}\leq0$ if and only if (\ref{Q-2}) holds.
In the case $\sigma_{I_2}\leq0$, the integral $I_2$, increasing in $R$, is bounded uniformly with respect to $R$. Applying the monotone convergence theorem, we conclude that $|\eta|_{\mathbb H}^{\gamma_1}|v|^px_1^\alpha$ is in $L^1\left(\Omega_2\right)$. Note that instead of the first inequality of (\ref{2-5}) we have, more precisely,
\begin{equation}
I_2\lesssim \left(\int_{\Omega_2\cap \Sigma_R}|\eta|_{\mathbb H}^{\gamma_1}|v|^px_1^\alpha \varphi^b d\eta\right)^{\frac{m_1m_2}{pq}}R^{\sigma_{I_2}}\lesssim \left(\int_{\Omega_2\cap \Sigma_R}|\eta|_{\mathbb H}^{\gamma_1}|v|^px_1^\alpha \varphi^b d\eta\right)^{\frac{m_1m_2}{pq}}.
\nonumber
\end{equation}
Finally, using the dominated convergence theorem, we obtain
$$\lim_{R\rightarrow+\infty}\int_{\Omega_2\cap \Sigma_R}|\eta|_{\mathbb H}^{\gamma_1}|v|^px_1^\alpha \varphi^b d\eta=0.$$
Therefore,
$$\lim_{R\rightarrow+\infty}I_2=0,$$
which implies that $v\equiv0$ in $\Omega_2$ and thus $u\equiv0$ in $\Omega_2$ via (\ref{es-5}).
The proof in the case $\sigma_{J_2}\leq0$ is analogous.
\end{proof}

\begin{proof}[Proof of Theorem \ref{thm3}] When $\Omega=\Omega_3$, we have for $j=1,2$,
$$K_{3,j}=\left(\int_{D_3}|\eta|_{\mathbb H}^{(1-\lambda_j)\gamma_j}\tau^\alpha\left|\Delta_{\mathbb
H}\varphi\right|^{\lambda_j} d\eta\right)^{\frac{1}{\lambda_j}},$$
and
$$L_{3,j}=\alpha\left(\int_{D_3}|\eta|_{\mathbb H}^{(1-\lambda_j)\gamma_j}\tau^{\alpha-\lambda_j}
\left|\nabla_{\mathbb H}\tau\right|^{\lambda_j}
\left|\nabla_{\mathbb H}\varphi\right|^{\lambda_j}
d\eta\right)^{\frac{1}{\lambda_j}}.$$
Then (\ref{est-IJ}) reads as
\begin{equation}\label{2-6}
 \begin{cases}
I_3^{1-\frac{m_1m_2}{pq}}\lesssim\left(K_{3,1}+ L_{3,1}\right)^{\frac{m_1}{q}}\left(K_{3,2}+  L_{3,2}\right),\\
J_3^{1-\frac{m_1m_2}{pq}}\lesssim \left(K_{3,1}+ L_{3,1}\right)\left(K_{3,2}+ L_{3,2}\right)
^{\frac{m_2}{p}}.
 \end{cases}
\end{equation}
We also use the first test function $\varphi(\eta)$ defined by (\ref{1-2}). By (\ref{1-3}), we obtain
\begin{equation}\label{2-7}
K_{3,j}\lesssim \left(R^{(1-\lambda_j)\gamma_j}R^{2\alpha} R^{-2\lambda_j}\int_{\Sigma_R}d\eta\right)^{\frac{1}{\lambda_j}}
\lesssim R^{\frac{(1-\lambda_j)\gamma_j+2\alpha+Q}{\lambda_j}-2}.
\end{equation}
By (\ref{2-3}), we obtain
\begin{equation}\label{2-8}
L_{3,j}\lesssim \left(R^{(1-\lambda_j)\gamma_j}R^{2(\alpha-\lambda_j)} R^{\lambda_j} R^{-\lambda_j}\int_{\Sigma_R}d\eta\right)^{\frac{1}{\lambda_j}}
\lesssim R^{\frac{(1-\lambda_j)\gamma_j+2\alpha+Q}{\lambda_j}-2}.
\end{equation}
Inserting (\ref{2-7}) and (\ref{2-8}) into (\ref{2-6}), we compute
\begin{equation}\label{2-9}
 \begin{cases}
I_3^{1-\frac{m_1m_2}{pq}}\lesssim R^{\sigma_{I_3}},\\
J_3^{1-\frac{m_1m_2}{pq}}\lesssim R^{\sigma_{J_3}},
 \end{cases}
\end{equation}
where
\begin{align}
\sigma_{I_3}:=\frac{(2\alpha+Q)\left(pq-m_1m_2\right)}{pq}
-\frac{m_1(p\gamma_2+m_2\gamma_1)}{pq}-\frac{2(m_1+q)}{q},\nonumber\\
\sigma_{J_3}:=\frac{(2\alpha+Q)\left(pq-m_1m_2\right)}{pq}
-\frac{m_2(q\gamma_1+m_1\gamma_2)}{pq}-\frac{2(m_2+p)}{p}.\nonumber
\end{align}

Note that $\sigma_{I_3}\leq0$ or $\sigma_{J_3}\leq0$ if and only if (\ref{Q-3}) holds.
In the case $\sigma_{I_3}\leq0$, the integral $I_3$, increasing in $R$, is bounded uniformly with respect to $R$. Applying the monotone convergence theorem, we conclude that $|\eta|_{\mathbb H}^{\gamma_1}|v|^p\tau^\alpha$ is in $L^1\left(\Omega_3\right)$. Note that instead of the first inequality of (\ref{2-9}) we have, more precisely,
\begin{equation}
I_3\lesssim \left(\int_{\Omega_3\cap \Sigma_R}|\eta|_{\mathbb H}^{\gamma_1}|v|^p\tau^\alpha \varphi^b d\eta\right)^{\frac{m_1m_2}{pq}}R^{\sigma_{I_3}}\lesssim \left(\int_{\Omega_3\cap \Sigma_R}|\eta|_{\mathbb H}^{\gamma_1}|v|^p\tau^\alpha \varphi^b d\eta\right)^{\frac{m_1m_2}{pq}}.
\nonumber
\end{equation}
Finally, using the dominated convergence theorem, we obtain
$$\lim_{R\rightarrow+\infty}\int_{\Omega_3\cap \Sigma_R}|\eta|_{\mathbb H}^{\gamma_1}|v|^p\tau^\alpha \varphi^b d\eta=0.$$
Therefore,
$$\lim_{R\rightarrow+\infty}I_3=0,$$
which implies that $v\equiv0$ in $\Omega_3$ and thus $u\equiv0$ in $\Omega_3$ via (\ref{es-5}).
The proof in the case $\sigma_{J_3}\leq0$ is analogous.
\end{proof}

%\section*{Acknowledgments}

\bibliographystyle{model1a-num-names}

\end{document}